\documentclass{amsproc}
\usepackage{amsmath}
\usepackage{amsfonts}
\usepackage{amssymb}
\usepackage{hyperref}
\hypersetup{colorlinks,
citecolor=black,
filecolor=black,
linkcolor=black,
urlcolor=black}
\usepackage{graphics,epstopdf}
\usepackage[pdftex]{graphicx}
\usepackage{newlfont}\newlength{\defbaselineskip}
\usepackage[left=1in,right=1in,top=1in,bottom=0.8in,footskip=0.25in]{geometry}
\usepackage{setspace}
\allowdisplaybreaks
\newtheorem{theorem}{Theorem}[section]
\newtheorem{example}{Example}[section]
\newtheorem{lemma}{Lemma}[section]
\newtheorem{definition}[theorem]{Definition}
\newtheorem{remark}{Remark}[section]

\numberwithin{equation}{section}

\usepackage{cite}

\usepackage{fancyhdr}
\usepackage{graphics}
\usepackage{color}
\usepackage{rotating}
\usepackage{fancybox}
\begin{document}
\title{Approximation of associated GBS operators by Sz$\acute{\text{a}}$sz-Mirakjan type operators
}
\maketitle
\begin{center}
{\bf Rishikesh Yadav$^{1,\dag}$,  Ramakanta Meher$^{1,\star}$,  Vishnu Narayan Mishra$^{2,\circledast}$}\\
$^{1}$Applied Mathematics and Humanities Department,
Sardar Vallabhbhai National Institute of Technology Surat, Surat-395 007 (Gujarat), India.\\
$^{2}$Department of Mathematics, Indira Gandhi National Tribal University, Lalpur, Amarkantak-484 887, Anuppur, Madhya Pradesh, India\\
\end{center}
\begin{center}
$^\dag$rishikesh2506@gmail.com,  $^\star$meher\_ramakanta@yahoo.com,
 $^\circledast$vishnunarayanmishra@gmail.com
\end{center}

\vskip0.5in

\begin{abstract}

In this article, the approximation properties of the Sz\'asz-Mirakjan type operators are studied for the function of two variables and  the rate of convergence of the bivariate operators is determined in terms of total and partial modulus of continuity. An associated GBS (Generalized Boolean Sum)-form of the bivariate Sz\'asz-Mirakjan type operators is considered for the function of two variables to find an approximation of $B$-continuous and $B$-differentiable function in the B$\ddot{\text{o}}$gel's space. Further, the degree of approximation of the GBS type operators is found in terms of mixed modulus of smoothness and functions belonging to the Lipschitz class as well as a pioneering result is obtained in terms of Peetre $K$-functional. Finally, the rate of convergence of the bivariate Sz\'asz-Mirakjan type operators and the associated GBS type operators are examined through graphical representation for the finite and infinite sum which shows that the rate of convergence of the associated GBS type operators is better than the bivariate Sz\'asz-Mirakjan type operators.



\end{abstract}
\subjclass \textbf{MSC 2010}: {41A25, 41A35, 41A36}.



\section{Introduction}
Approximation properties form an integral part in the study of approximation theory that includes convergence, rate of convergence, the order of approximation etc. Applications and convergence based discussion of the linear positive operators defined over different types of interval (finite or infinite) on $\mathbb{R}_+$, have been discussed by many researchers. In 1912, first of all, Bernstein proposed an operator, so-called Bernstein operator of one variable which approximates the functions defined over a finite interval $[0,1]$.  \\



In the study of \cite{g2, g3,g4,g6,g7,g8}, it is found that the Bernstein operators have been converted into bivariate Bernstein operators for function of two variables  over $[0,1]\times[0,1]$ with their graphical representation in the study of the approximation properties for the function of two variables.

Many results related to approximations theory have also been discussed by many authors \cite {MN,KMP,DSM,RS, RYM1}. Despite of these, if we move towards the operators defined over an infinite interval, first of all,  the Sz$\acute{\text{a}}$sz-Mirakjan operators were introduced and studied by  Mirakjan and Sz$\acute{\text{a}}$sz \cite{OS,GM} independently and  so many work done in a bivariate direction of these operators to generalize and check the behavior of the operatots for the function of two variables. Later on Sz$\acute{\text{a}}$sz-Mirakjan operators  have been discussed theoretically, numerically as well graphically by many authors \cite{g9,g10,g11,g12,YMM} using bivariate extension for approximation of the functions of two variables.  

Similarly, for the bivariate operators, one more property has been studied in B$\ddot{\text{o}}$gel space and that is the property of generalized boolean sum of the bivariate operators, so called GBS-type operators while the functions are considered to be $B$-continuous.
In 1934 and in 1935, B$\ddot{\text{o}}$gel  \cite{BK1,BK2} introduced B$\ddot{\text{o}}$gel space, after that Dobrescu and Matei \cite{new}, estimated the rate of convergence of associated GBS-type operators of the bivariate Bernstein operators in the  B$\ddot{\text{o}}$gel space.
Similarly, Badea et al\cite{BCB1}, proved the Korovkin type theorem for the function of two variables in the B$\ddot{\text{o}}$gel space. In 1988, Badea et al. \cite{BCB2} gave  a quantitative variant of Korovkin type theorem for $B$-continuous function and estimated the degree of approximation  of $B$-continuous function as well as $B$-differentiable function for certain linear positive operators. After that, in 1991, quantitative and non-quantitative Korovkin type theorem was proved by Badea and Cottin \cite{BCC} in the B$\ddot{\text{o}}$gel space.
Similarly, the approximation properties of the bivariate Bernstein type operators and their associated GBS operators have been examined by many reserchres (see \cite{BD1,BD2,MD1,PO1,PO2,17,18,19}).

In 2015, B$\breve{\text{a}}$rbosu and Muraru \cite{3} established some pioneer results through the associated GBS-type operators of Bernstein-Schurer-Stancu type operators using $q$-integers and some extension in terms of associated GBS-type operators of Kantorovich variant of a new kind of $q$-Bernstein-Schurer operators have been discussed by \cite{4}. Similarly, Sidharth et al. \cite{5} discussed the associated GBS-type operators. B$\breve{\text{a}}$rbosu et al. \cite{6} introduced GBS-Durrmeyer type operators based on $q$-integers, while \cite{7} has discussed the properties of GBS-type operators.
In 2016, Agrawal and Ispir \cite{1} estimated the degree of approximation of the Chlodowsky-Sz$\acute{\text{a}}$sz-Charlier type operators for the function of two variables. In the same year, Agrawal et al. \cite{2} studied the approximation properties and obtained the degree of approximation in terms of mixed modulus of smoothness.  

Yadav et al. \cite{RYV1} proposed bivariate Sz$\acute{\text{a}}$sz-Mirakjan type operators for the function of two variables, where they studied the approximation properties as well as the rate of convergence of the proposed bivariate operators in a polynomial weighted space and obtained a Voronovskaya type theorem as well as discussed simultaneous approximation property for the bivariate operators. The bivariate Sz$\acute{\text{a}}$sz-Mirakjan type operators are considered for continuous and bounded functions on $[0,\infty)\times[0,\infty)$ as below:
\begin{eqnarray}\label{R1}
 \Hat{Y}_{m,n,a}(f;x,y)&=& \sum\limits_{k_1=0}^{\infty}\sum\limits_{k_2=0}^{\infty}s_{m,n}^a(x,y) ~f\left(\frac{k_1}{m},\frac{k_2}{n} \right),
\end{eqnarray}
where
$s_{m,n}^a(x,y)=a^{\left(\frac{-x}{-1+a^{\frac{1}{m}}}\right)}a^{\left(\frac{-y}{-1+a^{\frac{1}{n}}}\right)}\frac{x^{k_1}y^{k_2}(\log{a})^{k_1+k_2}}{(-1+a^{\frac{1}{m}})^{k_1}(-1+a^{\frac{1}{n}})^{k_2}k_1!k_2!}$,
 $m,n\in\mathbb{N}$, $(x,y)\in X=[0,\infty)\times[0,\infty)$.\\
%

As an application, this work can also be extended into fraction difference sequence spaces. Maximum properties are studies of fraction difference operators in sequences spaces and some related research works are as \cite{BP1,BP2,BP3,BP4}.

The main purpose of this article is to investigate some results related to bivariate operators (\ref{R1}), like the order of approximation of the bivariate operators (\ref{R1}) in terms of total modulus of continuity and partial modulus of continuity. For further study of the proposed operators, the bivariate operators (\ref{R1}) are generalized into GBS (Generalized boolean Sum) form to determine the better rate of convergence than the proposed operators and establish the convergence properties of the GBS-type operators in the B$\ddot{\text{o}}$gel space with some approximations theorems in terms of  mixed modulus of smoothness with the aid of Lipschitz classes. The graphical and numerical approaches are presented to support the approximation results and made a comparison of bivariate operators (\ref{R1}) with its associated GBS-type operators in numerical sense. The best part of the this article is that, the graphical representation is shown for finite and infinite sum to determine the accuracy of the rate of convergence in its convergence behaviour. Finally, we have shown the comparison results of the GBS-type operators of the defined operators (\ref{R1}) with the GBS operators of  Mirakjan-Favard-Sz\'{a}sz to check the accuracy of the rate of convergence in terms of its numerical values. \\

For our main results, we need some basic lemmas. Consider the function $e_{ij}=x^{i}y^{j}$ such that $i,j\in\{0,1\}$ and $i+j\leq2$. Then the following lemma hold:

\begin{lemma}\label{L1} Let $x,y\geq 0$ and for each $m,n\in \mathbb{N}$. Then the following results hold:
\begin{eqnarray}
\Hat{Y}_{m,n,a}(e_{10};x,y)&=& \frac{x \log (a)}{m \left(a^{\frac{1}{m}}-1\right)}\\
\Hat{Y}_{m,n,a}(e_{01};x,y)&=& \frac{y \log (a)}{n \left(a^{\frac{1}{n}}-1\right)}\\
\Hat{Y}_{m,n,a}(e_{20};x,y)&=& \frac{x \log (a) \left(a^{\frac{1}{m}}+x \log (a)-1\right)}{m^2 \left(a^{\frac{1}{m}}-1\right)^2}\\
\Hat{Y}_{m,n,a}(e_{02};x,y)&=& \frac{y \log (a) \left(a^{\frac{1}{n}}+y \log (a)-1\right)}{n^2 \left(a^{\frac{1}{n}}-1\right)^2}.
\end{eqnarray}
\end{lemma}
\begin{proof}
Here, we have $x, y\geq 0$ and $m,n\in\mathbb{N}$ then
\begin{eqnarray*}
\Hat{Y}_{m,n,a}(e_{10};x,y)&=& \sum\limits_{k_1=0}^{\infty}\sum\limits_{k_2=0}^{\infty}s_{m,n}^a(x,y)\frac{k_1}{m}\\
&=& \sum\limits_{k_1=0}^{\infty}\sum\limits_{k_2=0}^{\infty} a^{\left(\frac{-x}{-1+a^{\frac{1}{m}}}\right)}a^{\left(\frac{-y}{-1+a^{\frac{1}{n}}}\right)}\frac{x^{k_1}y^{k_2}(\log{a})^{k_1+k_2}}{(-1+a^{\frac{1}{m}})^{k_1}(-1+a^{\frac{1}{n}})^{k_2}k_1!k_2!}\frac{k_1}{m}\\
&=& \frac{a^{\left(\frac{-x}{-1+a^{\frac{1}{m}}}\right)}a^{\left(\frac{-y}{-1+a^{\frac{1}{n}}}\right)}}{m}\sum\limits_{k_1=1}^{\infty}  \frac{x^{k_1}(\log{a})^{k_1}}{(-1+a^{\frac{1}{m}})^{k_1}(k_1-1)!}\cdot \sum\limits_{k_2=0}^{\infty}  \frac{y^{k_2}(\log{a})^{k_2}}{(-1+a^{\frac{1}{n}})^{k_2}k_2!}\\
&=& \frac{a^{\left(\frac{-x}{-1+a^{\frac{1}{m}}}\right)}a^{\left(\frac{-y}{-1+a^{\frac{1}{n}}}\right)}}{m}\cdot\frac{x\log{a}}{-1+a^{\frac{1}{m}}} \sum\limits_{k_1=1}^{\infty}  \frac{x^{k_1-1}(\log{a})^{k_1-1}}{(-1+a^{\frac{1}{m}})^{k_1-1}(k_1-1)!}\cdot\left( e^{\log{a}}\right)^{\frac{y}{-1+a^{\frac{1}{n}}}}.\\
&=& \frac{x\log{a}}{\left(-1+a^{\frac{1}{m}}\right)m}\\
\Hat{Y}_{m,n,a}(e_{20};x,y)&=& \sum\limits_{k_1=0}^{\infty}\sum\limits_{k_2=0}^{\infty}s_{m,n}^a(x,y)\left(\frac{k_1}{m}\right)^2\\
&=& a^{\left(\frac{-x}{-1+a^{\frac{1}{m}}}\right)}a^{\left(\frac{-y}{-1+a^{\frac{1}{n}}}\right)}\sum\limits_{k_1=2}^{\infty}  \frac{k_1x^{k_1}(\log{a})^{k_1}}{(-1+a^{\frac{1}{m}})^{k_1}(k_1-1)(k_1-2)!}\cdot \sum\limits_{k_2=0}^{\infty}  \frac{y^{k_2}(\log{a})^{k_2}}{(-1+a^{\frac{1}{n}})^{k_2}k_2!}\\
&=& a^{\left(\frac{-x}{-1+a^{\frac{1}{m}}}\right)}\frac{1}{m^2}\sum\limits_{k_1=2}^{\infty}  \frac{x^{k_1}(\log{a})^{k_1}}{(-1+a^{\frac{1}{m}})^{k_1}(k_1-2)!}\left(1+\frac{1}{k_1-1} \right)\\
&=&  \left(\frac{x\log{a}}{\left(-1+a^{\frac{1}{m}}\right)m}\right)^2+\frac{x\log{a}}{\left(-1+a^{\frac{1}{m}}\right)m^2}.
\end{eqnarray*}
Similarly, we can prove other results.
\end{proof}

\begin{lemma}\label{L2}
For every $x,y\in X=[0,\infty)\times[0,\infty)$ and $m,n\in\mathbb{N}$, it gives the following results:

\begin{eqnarray*}
1.~\Hat{Y}_{m,n,a}((t-x);x,y)&=&-\frac{x \left(m a^{\frac{1}{m}}-\log (a)-m\right)}{m \left(a^{\frac{1}{m}}-1\right)}\\
2.~\Hat{Y}_{m,n,a}((s-y);x,y)&=& -\frac{y \left(n a^{\frac{1}{n}}-\log (a)-n\right)}{n \left(a^{\frac{1}{n}}-1\right)}\\
3.~\Hat{Y}_{m,n,a}((t-x)^2;x,y)&=&\frac{x \left(m^2 x \left(a^{\frac{1}{m}}-1\right)^2-\left(a^{\frac{1}{m}}-1\right) \log (a) (2 m x-1)+x (\log{a})^2\right)}{m^2 \left(a^{\frac{1}{m}}-1\right)^2}\\
4.~\Hat{Y}_{m,n,a}((s-y)^2;x,y)&=&\frac{y \left(n^2 y \left(a^{\frac{1}{n}}-1\right)^2-\left(a^{\frac{1}{n}}-1\right) \log (a) (2 n y-1)+y (\log{a})^2\right)}{n^2 \left(a^{\frac{1}{n}}-1\right)^2}\\
5.~\Hat{Y}_{m,n,a}((t-x)^4;x,y)&=&\frac{x}{m^4 \left(a^{\frac{1}{m}}-1\right)^4}\Bigg\{m^4 x^3 \left(a^{\frac{1}{m}}-1\right)^4-\left(a^{\frac{1}{m}}-1\right)^3 (-1 + 4 m x - 6 m^2 x^2 + 4 m^3 x^3) \log{a}\\&& +\left(a^{\frac{1}{m}}-1\right)^2 x (7 - 12 m x + 6 m^2 x^2) (\log{a})^2 \\&&
-2 \left(a^{\frac{1}{m}}-1\right) x^2 (-3 + 2 m x) (\log{a})^3 + x^3 (\log{a})^4\Bigg\} \\
6.~\Hat{Y}_{m,n,a}((s-y)^4;x,y)&=&\frac{y}{n^4 \left(a^{\frac{1}{n}}-1\right)^4}\Bigg\{n^4 y^3 \left(a^{\frac{1}{n}}-1\right)^4-\left(a^{\frac{1}{n}}-1\right)^3 (-1 + 4 n y - 6 n^2 y^2 + 4 n^3 y^3) \log{a}\\&& +\left(a^{\frac{1}{n}}-1\right)^2 y (7 - 12 n y + 6 n^2 y^2) (\log{a})^2 \\&&
-2 \left(a^{\frac{1}{n}}-1\right) y^2 (-3 + 2 n y) (\log{a})^3 + y^3 (\log{a})^4\Bigg\}.
\end{eqnarray*}
\end{lemma}

\begin{proof}
Using the Lemma \ref{L1}, for every $x,y \in X$ and for all $m,n\in\mathbb{N}$, we have
\begin{eqnarray*}
1.~\Hat{Y}_{m,n,a}((t-x);x,y)&=& \sum\limits_{k_1=0}^{\infty}\sum\limits_{k_2=0}^{\infty}s_{m,n}^a(x,y) \left(\frac{k_1}{m}-x \right)\\
&=&  \sum\limits_{k_1=0}^{\infty}\sum\limits_{k_2=0}^{\infty}s_{m,n}^a(x,y) \frac{k_1}{m}-x \sum\limits_{k_1=0}^{\infty}\sum\limits_{k_2=0}^{\infty}s_{m,n}^a(x,y)\\
&=& \frac{x\log{a}}{\left(-1+a^{\frac{1}{m}}\right)m}-xa^{\left(\frac{-x}{-1+a^{\frac{1}{m}}}\right)}a^{\left(\frac{-y}{-1+a^{\frac{1}{n}}}\right)} \left( e^{\log{a}}\right)^{\frac{x}{-1+a^{\frac{1}{m}}}} \left( e^{\log{a}}\right)^{\frac{y}{-1+a^{\frac{1}{n}}}}\\
&=& \frac{x\log{a}}{\left(-1+a^{\frac{1}{m}}\right)m}-x.\\
3.~\Hat{Y}_{m,n,a}((t-x)^2;x,y)&=& \Hat{Y}_{m,n,a}((t^2-2tx+x^2);x,y)\\
&=& \Hat{Y}_{m,n,a}(e_{20};x,y)-2x\Hat{Y}_{m,n,a}(e_{10};x,y)+x^2\\
&=& \frac{x \log (a) \left(a^{\frac{1}{m}}+x \log (a)-1\right)}{m^2 \left(a^{\frac{1}{m}}-1\right)^2}-2x\left(\frac{x \log (a)}{m \left(a^{\frac{1}{m}}-1\right)} \right)-2x\\
&=& \frac{x \left(m^2 x \left(a^{\frac{1}{m}}-1\right)^2-\left(a^{\frac{1}{m}}-1\right) \log (a) (2 m x-1)+x (\log{a})^2\right)}{m^2 \left(a^{\frac{1}{m}}-1\right)^2}.
\end{eqnarray*}
Similarly, other equalities can be proved.
\end{proof}

\begin{lemma} \label{L3}
For all $x,y\geq 0$, the following inequalities hold true:
\begin{eqnarray*}
\Hat{Y}_{m,n,a}((t-x)^2;x,y)&\leq & \frac{x(x+1)}{m}=\delta_m^{'2}(x)\\
\Hat{Y}_{m,n,a}((s-y)^2;x,y)&\leq & \frac{y(y+1)}{m}=\delta_n^{'2}(y).
\end{eqnarray*}
\end{lemma}
\begin{proof}
For all $m,n\in\mathbb{N}$ and $x\geq 0$, we have
\begin{eqnarray*}
\Hat{Y}_{m,n,a}((t-x)^2;x,y)&=&  \frac{x \left(m^2 x \left(a^{\frac{1}{m}}-1\right)^2-\left(a^{\frac{1}{m}}-1\right) \log (a) (2 m x-1)+x (\log{a})^2\right)}{m^2 \left(a^{\frac{1}{m}}-1\right)^2}\\
&=&x\left( x\left(\frac{\log{a}}{m\left( a^{\frac{1}{m}}-1\right)} \right)^2-\frac{2x\log{a}}{m\left( a^{\frac{1}{m}}-1\right)}+x+\frac{\log{a}}{m^2\left( a^{\frac{1}{m}}-1\right)}\right)\\
&=& x\left( x\left(\frac{\log{a}}{m\left( a^{\frac{1}{m}}-1\right)} -1\right)^2+\frac{\log{a}}{m^2\left( a^{\frac{1}{m}}-1\right)}\right)\\
&\leq & x\left( \frac{x}{m}+\frac{1}{m}\right)=\frac{x(x+1)}{m}.
\end{eqnarray*}
Similarly, other inequality can be proved.
\end{proof}

\begin{remark}
For all $(x,y)\in [0,c]\times[0,d]$, where $0\leq x\leq c$ and $0\leq y\leq d$, we have 
\begin{eqnarray}
\Hat{Y}_{m,n,a}((t-x)^2;x,y)&\leq & \frac{c(c+1)}{m}=\frac{\lambda_x}{m}\label{re2},\\
\Hat{Y}_{m,n,a}((s-y)^2;x,y)&\leq & \frac{d(d+1)}{n}=\frac{\lambda_y}{n}\label{re3},
\end{eqnarray}
where $\lambda_x, \lambda_y$ are  positive constants.
\end{remark}
\begin{proof}
Using the Lemma \ref{L3}, we can obtain the required results.
\end{proof}

\begin{lemma}\label{L5}
For all $x,y\in [0,c]\times[0,d]$ and $m,n\in\mathbb{N}$, the following inequalities hold true 
\begin{eqnarray}
\Hat{Y}_{m,n,a}((t-x)^4;x,y)\leq \frac{M_x}{m^2},
\end{eqnarray}
\begin{eqnarray}
\Hat{Y}_{m,n,a}((s-y)^4;x,y)\leq \frac{M_y}{n^2},
\end{eqnarray}
where $M_x, M_y$ are positive constants. 
\end{lemma}
\begin{proof}
By Lemma \ref{L2}, we have 
\begin{eqnarray*}
\Hat{Y}_{m,n,a}((t-x)^4;x,y)&=& x\Bigg(\frac{\log{a}}{m^4\left(a^{\frac{1}{m}}-1\right)} \Bigg)+x^2 \frac{\log{a}}{m\left(a^{\frac{1}{m}}-1\right)}\Bigg(\frac{7\log{a}}{m^3\left(a^{\frac{1}{m}}-1\right)} -\frac{4}{m^2}\Bigg)\\
&& +\frac{6x^3\log{a}}{m\left(a^{\frac{1}{m}}-1\right)}\Bigg( \frac{1}{m}-\frac{2\log{a}}{m^2\left(a^{\frac{1}{m}}-1\right)}+\frac{(\log{a})^2}{m^3\left(a^{\frac{1}{m}}-1\right)^2} \Bigg)\\
&&+x^4\Bigg(1- \frac{4\log{a}}{m\left(a^{\frac{1}{m}}-1\right)}+\frac{6(\log{a})^2}{m^2\left(a^{\frac{1}{m}}-1\right)^2}-\frac{4(\log{a})^3}{m^3\left(a^{\frac{1}{m}}-1\right)^3}+\frac{(\log{a})^4}{m^4\left(a^{\frac{1}{m}}-1\right)^4}\Bigg)\\
&\leq & \frac{x}{n^3}+\frac{x^2}{n^2}\left(\frac{7\log{a}}{m\left(a^{\frac{1}{m}}-1\right)}-4 \right)+\frac{6x^3}{n}\left(\frac{\log{a}}{m\left(a^{\frac{1}{m}}-1\right)}-1 \right)^2+x^4\left(\frac{\log{a}}{m\left(a^{\frac{1}{m}}-1\right)}-1 \right)^4\\
&\leq & \frac{x}{m^3}+\frac{7x^2}{m^2}\left(\frac{1}{m}-\frac{3}{7} \right)+\frac{6x^3}{m^2}+\frac{x^4}{m^3}\\
&=&\frac{1}{m^3}(x^4+7x^2+x)+\frac{1}{m^2}(3x^2+6x^3)\\
&\leq & \frac{1}{m^2}(x^4+10x^2+6x^3+x)\\
&\leq & \frac{1}{m^2}(c^4+10c^2+6c^3+c)=\frac{M_x}{m^2}.
\end{eqnarray*}
Similarly, it can be proved that 
\begin{eqnarray*}
\Hat{Y}_{m,n,a}((s-y)^4;x,y)\leq \frac{M_y}{n^2}.
\end{eqnarray*}
\end{proof}

\section{Basic properties of the bivariate operators}
For finding the rate of convergence of the bivariate operators defined by (\ref{R1}) in terms of modulus of continuity, here we define the modulus of continuity. Let the function $f(x,y)\in C_B(X=[0,\infty)\times[0,\infty))$, be the space of all continuous and bounded function defined on $X=[0,\infty)\times[0,\infty)$, then the total (complete) modulus of continuity for the function of two variables can be defined as:
\begin{eqnarray}
\omega (f,\delta)=\sup\{|f(t,s)-f(x,y)|: \sqrt{(t-x)^2+(s-y)^2}\leq\delta,~(t,s)\in X,~\delta>0\} 
\end{eqnarray}
and the partial modulus of continuity can be defined as \cite{MI}:
\begin{eqnarray}
\omega_1 (f,\delta)&=&\sup\{|f(u_1,y)-f(u_2,y)|: |u_1-u_2|\leq\delta,~\delta>0\},\label{p1}\\
\omega_2 (f,\delta)&=&\sup\{|f(x,v_1)-f(x,v_2)|: |v_1-v_2|\leq\delta,~\delta>0\}.\label{p2}
\end{eqnarray}

The following theorem will show the rate of convergence of the bivariate operators (\ref{R1}) with the help of modulus of continuity.

\begin{theorem}
If the bivariate operators $\Hat{Y}_{m,n,a}(f;x,y)$ defined by (\ref{R1}) are linear and positive, then the following relations hold: 
\begin{eqnarray}
|\Hat{Y}_{m,n,a}(f;x,y)-f(x,y)|&\leq & 2\omega (f;\delta_{m,n}),\\
|\Hat{Y}_{m,n,a}(f;x,y)-f(x,y)|&\leq & 2\{\omega_1 (f,\delta_m)+\omega_2 (f,\delta_n)\},
\end{eqnarray}
where $\omega$ is the total modulus of continuity and $\omega_1, \omega_2$ are the partial modulus of continuity with respect to $x,y$ respectively.
\end{theorem}
\begin{proof}
Using the definition of modulus of continuity, we can write 
\begin{eqnarray*}
|\Hat{Y}_{m,n,a}(f;x,y)-f(x,y)|&\leq &  \Hat{Y}_{m,n,a}(|f(t,s)-f(x,y)|;x,y))\\ &\leq & \Hat{Y}_{m,n,a}(\omega(\sqrt{(t-x)^2+(s-y)^2};x,y))\\ &\leq & \omega(f;\delta)\left(1+\frac{1}{\delta} (\Hat{Y}_{m,n,a}(\sqrt{(t-x)^2+(s-y)^2};x,y)) \right)\\&\leq &  \omega(f;\delta)\left(1+\frac{1}{\delta} \{\Hat{Y}_{m,n,a}((t-x)^2+(s-y)^2;x,y)\}^{\frac{1}{2}} \right)\\
&=&\omega(f;\delta)\Bigg\{1+\frac{1}{\delta}\Bigg(\frac{x \left(m^2 x \left(a^{\frac{1}{m}}-1\right)^2-\left(a^{\frac{1}{m}}-1\right) \log (a) (2 m x-1)+x (\log{a})^2\right)}{m^2 \left(a^{\frac{1}{m}}-1\right)^2}\\
&&+\frac{y \left(n^2 y \left(a^{\frac{1}{n}}-1\right)^2-\left(a^{\frac{1}{n}}-1\right) \log (a) (2 n y-1)+y (\log{a})^2\right)}{n^2 \left(a^{\frac{1}{n}}-1\right)^2} \Bigg)^{\frac{1}{2}}\Bigg\},
\end{eqnarray*}
upon considering 
\begin{eqnarray*}
\delta &=&\Bigg\{1+\frac{1}{\delta}\Bigg(\frac{x \left(m^2 x \left(a^{\frac{1}{m}}-1\right)^2-\left(a^{\frac{1}{m}}-1\right) \log (a) (2 m x-1)+x (\log{a})^2\right)}{m^2 \left(a^{\frac{1}{m}}-1\right)^2}\\
&&+\frac{y \left(n^2 y \left(a^{\frac{1}{n}}-1\right)^2-\left(a^{\frac{1}{n}}-1\right) \log (a) (2 n y-1)+y (\log{a})^2\right)}{n^2 \left(a^{\frac{1}{n}}-1\right)^2} \Bigg)^{\frac{1}{2}}\Bigg\}=\delta_{m,n},
\end{eqnarray*}
the next one step will give  the required result.\\
Now to prove the second part of this theorem. Upon  using the properties (\ref{p1}), (\ref{p2}) and with the help of Cauchy-Schwarz inequality, we get:
\begin{eqnarray}\label{I1}
\nonumber|\Hat{Y}_{m,n,a}(f;x,y)-f(x,y)|&\leq &  \Hat{Y}_{m,n,a}(|f(t,s)-f(x,y)|;x,y))\\ \nonumber &\leq & \Hat{Y}_{m,n,a}(|f(t,s)-f(t,y)|;x,y))+\Hat{Y}_{m,n,a}(|f(t,y)-f(x,y)|;x,y))\\\nonumber&\leq & \omega_2(f,\delta_m) \left(1+\frac{1}{\delta_m} (\Hat{Y}_{m,n,a}(\sqrt{(s-y)^2};x,y))^{\frac{1}{2}} \right)\\
&&+ \omega_1(f,\delta_n) \left(1+\frac{1}{\delta_n} (\Hat{Y}_{m,n,a}(\sqrt{(t-x)^2};x,y))^{\frac{1}{2}} \right),
\end{eqnarray}
where
\begin{eqnarray}
\delta_m &=&\frac{x \left(m^2 x \left(a^{\frac{1}{m}}-1\right)^2-\left(a^{\frac{1}{m}}-1\right) \log (a) (2 m x-1)+x (\log{a})^2\right)}{m^2 \left(a^{\frac{1}{m}}-1\right)^2}\\
\delta_n &=& \frac{y \left(n^2 y \left(a^{\frac{1}{n}}-1\right)^2-\left(a^{\frac{1}{n}}-1\right) \log (a) (2 n y-1)+y (\log{a})^2\right)}{n^2 \left(a^{\frac{1}{n}}-1\right)^2}
\end{eqnarray}
hence, by using inequality (\ref{I1}), the above result can be obtained.
\end{proof}

\subsection{Some basic definitions for associated GBS (Generalized Boolean Sum) operators} 
In recent years, the study of generalized Boolean sum (GBS) operators of certain linear positive operators is an interesting topic in approximation theory and function theory. 
In order to make analysis in multidimensional spaces, Karl B$\ddot{\text{o}}$gel introduced the concepts of  $B$-continuous and $B$-differentiable function in \cite{BK1,BK2}. 
In \cite{BCGH1}, the authors discussed some significance role of the B$\ddot{\text{o}}$gel space. They proved that the space of all bounded B$\ddot{\text{o}}$gel functions is isometrically isometric with the completion of the blending function space with respect to suitable norm. 
Also the main importance of the B$\ddot{\text{o}}$gel space is that the functions which are not continuous in general but are $B$-continuous can also be approximated by the operators.
In this subsection, some basic definitions are defined for associated GBS-type operators in the B$\ddot{\text{o}}$gel space and their related properties are discussed. 

\begin{definition}
\textbf{B-Continuous:} Consider two compact intervals $\mathfrak{A_1}, \mathfrak{A_2}\subset \mathbb{R}$, a function $f:\mathfrak{A_1}\times\mathfrak{A_2}\to\mathbb{R}$ is said to be $B$-continuous function at a point $(u_0,v_0)\in \mathfrak{A_1}\times\mathfrak{A_2}$, if  
\begin{eqnarray}\label{e1}
\underset{(u,v)\to(u_0,v_0)}\lim \Delta f((u,v),(u_0,v_0))=0,
\end{eqnarray}
where $\Delta f((u,v),(u_0,v_0))=f(u,v)-f(u,v_0)-f(u_0,v)+f(u_0,v_0)$ and  the set of all $B$-continuous function is denoted by $C_b(\mathfrak{A_1}\times \mathfrak{A_2})$. 
\end{definition}
\begin{definition}
\textbf{$B$-Bounded:} A real valued function $f$ defined on $\mathfrak{A_1}\times\mathfrak{A_2}$ is said to be $B$-Bounded, if there exist a positive constant $\mathcal{M}$ such that
\begin{eqnarray}\label{e2}
\Delta f((u,v),(u_0,v_0))\leq \mathcal{M},
\end{eqnarray}
i.e. denoted by $B_b(\mathfrak{A_1}\times\mathfrak{A_2})$.
\end{definition}
\begin{definition}
\textbf{$B$-Differentiable:} A function $f$ is called $B$-Differentiable iff 
\begin{eqnarray}\label{e3}
D_Bf(u_0,v_0)=\underset{(u,v)\to(u_0,v_0)}\lim \frac{\Delta f((u,v),(u_0,v_0))}{(u-u_0)(v-v_0)},
\end{eqnarray}
provided the limit exists and finite where the set of all $B$-differentiable  functions is denoted by $D_b(\mathfrak{A_1}\times\mathfrak{A_2})$. For more details see \cite{BK1,BK2}.
\end{definition}

Motivated by cited papers in introduction part, here, we define the associated GBS-type operators of the above defined biavriate operators (\ref{R1}) to investigate their approximation  properties in the B$\ddot{\text{o}}$gel space. The main motive of this part is to determine the convergence results of the GBS-type operators defined by (\ref{R2}) along with their properties by theoretical, numerical as well as graphical sense. The goodness of the GBS-type operators is that these operators have a better rate of convergence than the proposed bivariate operators (\ref{R1}) as well as the GBS operators of Mirakjan-Favard-Sz\'{a}sz. So, before the discussion of their properties, first we construct here the GBS-type operators of the above bivariate operators (\ref{R1}).\\

Consider two compact intervals $\mathfrak{A_1}, \mathfrak{A_2}\subset \mathbb{R}$ and for any point $(x,y)\in \mathfrak{A_1}\times\mathfrak{A_2}$, the Boolean sum of the function $f:\mathfrak{A_1}\times\mathfrak{A_2}\to\mathbb{R}$ can be defined as $\Delta f((x,y),(t,s))=f(x,y)-f(x,t)-f(s,y)+f(t,s)$ at a point $(t,s)\in \mathfrak{A_1}\times\mathfrak{A_2}$. Then the associated GBS (Generalized Boolean Sum)-type operators of $\Hat{Y}_{m,n,a}(f;x,y)$ can be expressed as
\begin{eqnarray}\label{R2}
\nonumber \Hat{{BY}}_{m,n}^{a}(f;x,y)&=& \Hat{Y}_{m,n,a}(f(x,s)+f(t,y)-f(t,s))\\
&=& \sum\limits_{k_1=0}^{\infty}\sum\limits_{k_2=0}^{\infty}s_{m,n}^a(x,y) \Bigg(f\left(x,\frac{k_2}{n} \right)+f\left(\frac{k_1}{m},y \right)-f\left(\frac{k_1}{m},\frac{k_2}{n} \right) \Bigg),
\end{eqnarray}
where $f\in C_b(X_b=[0,c]\times[0,d])$.

\subsection{Degree of the approximation of the GBS-type operators}
In this subsection, we discuss the rate of convergence of the GBS-type operators with the help of modulus of smoothness in a B$\ddot{\text{o}}$gel space, and get a relation using the mixed modulus of smoothness. Now to define the modulus of smoothness, we assume that the function $f\in C_b(X_b=[0,c]\times[0,d])$. The property of mixed modulus of smoothness is same as the modulus of continuity, which can be defined as
\begin{eqnarray}
\omega_B(f;\delta_1,\delta_2)=\sup\{|\Delta f(t,s;x,y)|:|t-x|<\delta_1, |s-y|<\delta_2,~~ (x,y),(t,s)\in X_b=[0,c]\times[0,d]\},
\end{eqnarray}
for any $(\delta_1,\delta_2)\in X=[0,\infty)\times[0,\infty)$ and having property
\begin{eqnarray}
\omega_B (f;\delta_m,\delta_n)\to 0,~ \text{as}~m,n\to\infty.
\end{eqnarray}
\begin{remark}\label{re1}
The property of the modulus of smoothness can be defined as:
\begin{eqnarray}
\omega_B(f;\mu_1\delta_1,\mu_2\delta_2)= (1+\mu_1)(1+\mu_2)\omega_B(f;\delta_1,\delta_2),~~~\mu_1,\mu_2>0.
\end{eqnarray}
\end{remark}
\begin{theorem}
Let $f\in C_b(X_b)$ and $\Hat{{BY}}_{m,n}^{a}(f;x,y)$ be linear  positive operators defined by (\ref{R2}), then the following inequality holds:
\begin{eqnarray}
|\Hat{{BY}}_{m,n}^{a}f)(x,y)-f(x,y)|\leq 4 \omega_B(f;\delta_m',\delta_n').
\end{eqnarray}
\end{theorem}
\begin{proof}
Upon using the Remark (\ref{re1}), it can be written as:
\begin{eqnarray*}
\nonumber|\Delta f(t,s;x,y)| &\leq & \omega_B(f;\delta_1,\delta_2)\\
&\leq& \Bigg(1-\frac{|t-x|}{\delta_1}\Bigg)\Bigg(1-\frac{|s-y|}{\delta_2}\Bigg)\omega_B(f;\delta_1,\delta_2),~~\delta_1,\delta_2\geq0.
\end{eqnarray*}
Using the property of the difference function $\Delta f(t,s;x,y)$ and applying the operators (\ref{R1}), it gives 
\begin{eqnarray}
\Hat{{BY}}_{m,n}^{a}(f;x,y)&=& f(x,y)\Hat{Y}_{m,n,{a}}(1,x,y)-\Hat{Y}_{m,n,{a}}(\Delta f(t,s;x,y),x,y),
\end{eqnarray}
Upon using Cauchy-Schwartz inequality in equation (2.17),  it obtains
\begin{eqnarray*}
|\Hat{{BY}}_{m,n}^{a}(f;x,y)-f(x,y)|&\leq & \Hat{Y}_{m,n,{a}}(|\Delta f(t,s;x,y)|;x,y)\\
&\leq & \Bigg(\Hat{Y}_{m,n,{a}}(e_{00};x,y)+\frac{1}{\delta_1}\Hat{Y}_{m,n,{a}}(|t-x|;x,y) \Bigg)\\
&&\times\Bigg(\Hat{Y}_{m,n,{a}}(e_{00};x,y)+\frac{1}{\delta_2}\Hat{Y}_{m,n,{a}}(|s-y|;x,y) \Bigg)\omega_B(f;\delta_1,\delta_2)\\
&\leq & \Bigg(1+\frac{1}{\delta_1}\sqrt{\Hat{Y}_{m,n,{a}}((t-x)^2;x,y)}+ \frac{1}{\delta_2}\sqrt{\Hat{Y}_{m,n,{a}}((s-y)^2;x,y)}\\
&&+\frac{1}{\delta_1\delta_2}\sqrt{\Hat{Y}_{m,n,{a}}((t-x)^2;x,y)}\sqrt{\Hat{Y}_{m,n,{a}}((s-y)^2;x,y)} \Bigg)\omega_B(f;\delta_1,\delta_2).
\end{eqnarray*}
Now by using Lemma \ref{L3}, and choosing
$\delta_1=\delta_m^{'}, \delta_2=\delta_n^{'} $, the desired results can be obtained.       
\end{proof}

Next we will find the degree of approximation of the GBS-type operators defined by (\ref{R2}), by means of $B$-continuous function belonging to the Lipschitz class and it can be defined as:
\begin{eqnarray}
\text{Lip}_M(\mu_1,\mu_2)=\{f\in C_b(X_b):|\Delta f((u,v),(u_0,v_0))|\leq M |u-u_0|^{\mu_1}|v-v_0|^{\mu_2},~ \mu_1, \mu_2\in (0,1]\},
\end{eqnarray}
where $(u,v),(u_0,v_0)\in X_b$ and $M>0$.
\begin{theorem}
Let $f\in \text{Lip}_M(\mu_1,\mu_2)$, then there exist a positive constant $M$, such that  
\begin{eqnarray}
|\Hat{{BY}}_{m,n}^{a}(f;x,y)-f(x,y)|\leq M \delta_m'^{\frac{\mu_1}{2}}\delta_n'^{\frac{\mu_2}{2}},
\end{eqnarray}
where $\delta_m'=\sqrt{\frac{x(x+1)}{m}},~\delta_n'=\sqrt{\frac{y(y+1)}{n}}$.
\end{theorem}

\begin{proof}
By using the linearity property of GBS-type operators  (\ref{R2}) and by definition of $\Hat{{BY}}_{m,n}^{a}(f;x,y)$, we can write
\begin{eqnarray*}
|\Hat{{BY}}_{m,n}^{a}(f;x,y)-f(x,y)|&\leq & \Hat{Y}_{m,n,a}(|\Delta f((t,s),(x,y))|;x,y)\\
&\leq & M \Hat{Y}_{m,n,a}(|u-u_0|^{\mu_1}|v-v_0|^{\mu_2};x,y)\\
&=& M\Hat{Y}_{m,n,a}(|t-x|^{\mu_1};x,y)\Hat{Y}_{m,n,a}(|s-y|^{\mu_2};x,y),
\end{eqnarray*}
Using H$\ddot{\text{o}}$lder's inequality and by considering $l_1=\frac{2}{\mu_1},~r_1=\frac{2}{2-\mu_1}$ and $l_2=\frac{2}{\mu_2},~r_2=\frac{2}{2-\mu_2}$ , in the next step, the required result can be obtained as,
\begin{eqnarray*}
|\Hat{{BY}}_{m,n}^{a}(f;x,y)-f(x,y)|&\leq & M(\Hat{Y}_{m,n,a}((t-x)^{2};x,y))^{\frac{\mu_1}{2}}(\Hat{Y}_{m,n,a}((s-y)^{2};x,y))^{\frac{\mu_2}{2}}\\
&\leq &M \delta_m'^{\frac{\mu_1}{2}}\delta_n'^{\frac{\mu_2}{2}}.
\end{eqnarray*}
Hence proved.
\end{proof}
Next the rate of convergence of the above operators can be found, when the function is $B$-differentiable, and it is defined by  (\ref{e3}). 

\begin{lemma}
For any $x,y\geq 0$ and for all $m,n\in \mathbb{N}$, we have
\begin{eqnarray*}
\Hat{Y}_{m,n,a}((\cdot-x)^{2i}(\star-y)^{2j};x,y)=\Hat{Y}_{m,n,a}((\cdot-x)^{2i});x,y)\Hat{Y}_{m,n,a}(\star-y)^{2j};x,y),~~\forall~i,j\in\mathbb{N}\cup\{0\}.
\end{eqnarray*}
\end{lemma}
\begin{proof}
Given that $x,y\geq 0$ and $n,m\in\mathbb{N}$ then, we have 

\begin{eqnarray*}
\Hat{Y}_{m,n,a}((\cdot-x)^{2i}(\star-y)^{2j};x,y)&=&\sum\limits_{k_1=0}^{\infty}\sum\limits_{k_2=0}^{\infty}s_{m,n}^a(x,y)\left(\frac{k_1}{m}-x \right)^{2i}\left(\frac{k_2}{n}-y \right)^{2j}\\
&=&\sum \limits_{k_1=0}^{\infty} s_{m}^a(x,y)\left(\frac{k_1}{m}-x \right)^{2i}\sum\limits_{k_2=0}^{\infty}s_{m}^a(x,y) \left(\frac{k_2}{n}-y \right)^{2j}\\&=&\Hat{Y}_{m,n,a}((\cdot-x)^{2i});x,y)\Hat{Y}_{m,n,a}(\star-y)^{2j};x,y).
\end{eqnarray*}
Hence proved.
\end{proof}

\begin{theorem}
Let $f\in D_b(X_b)$ and $D_B f\in B_b(X_b)$, then there exist a positive constant $M_1$, such that 
\begin{eqnarray}
|\Hat{{BY}}_{m,n}^{a}(f;x,y)-f(x,y)|\leq \frac{M_4}{\sqrt{mn}}	\Bigg\{ 3M_3\|D_B f\|+\omega_B\left( D_B f;\frac{1}{\sqrt{m}},\frac{1}{\sqrt{n}}\right) \Bigg\} 
\end{eqnarray}
\end{theorem}
\begin{proof}
Using mean value theorem for $B$-differntaible functions, it can be written as 
\begin{eqnarray}\label{e5}
 D_Bf(\beta,\gamma)=\frac{\Delta f((t,s),(x,y))}{(t-x)(s-y)},~~\text{where} ~\beta\in (t,x);~\gamma\in (s,y).
\end{eqnarray}
By using the property of $\Delta f((t,s),(x,y))$, it gives:
\begin{eqnarray}\label{e4}
D_B f(\beta,\gamma)=\Delta D_B f((\beta,\gamma),(x,y))+D_B f(x,\gamma)+D_B f(\beta,y)-D_B f(x,y),
\end{eqnarray}
since, $D_B f\in B_b(X_b)$, so by using equation (\ref{e4}) and the equality (\ref{e5}), it obtains
\begin{eqnarray*}
|\Hat{Y}_{m,n,a}(\Delta f((t,s),(x,y));x,y)|&=&|\Hat{Y}_{m,n,a}((t-x)(s-y)D_B f(\beta,\gamma));x,y)|\\
&\leq & \Hat{Y}_{m,n,a}(|(t-x)||(s-y)|| \Delta D_B f(\beta,\gamma),(x,y))|;x,y)\\
&&+\Hat{Y}_{m,n,a}(|t-x||s-y|(|D_B f(x,\gamma)|+|D_B f(\beta,y)|-|D_B f(x,y)|);x,y)\\
&\leq & \Hat{Y}_{m,n,a}(|t-x||y-s|\omega_B(D_B f;|\beta-x|,|\gamma-y|);x,y)\\
&&+3\|D_B f\|\Hat{Y}_{m,n,a}(|t-x||y-s|;x,y),
\end{eqnarray*}
as $\beta\in (x,t)$ and $\gamma\in (y,s)$ (already assumed) and with the property of modulus, for $h_m,h_n>0$, we have
\begin{eqnarray*}
\omega_B(D_B f;|\beta-x|,|\gamma-y|)&\leq & \omega_B(D_B f;|t-x|,|s-y|)\\
&\leq & \left(1+\frac{|t-x|}{h_m}\right)\left(1+\frac{|s-y|}{h_n}\right)\omega_B(D_B f;h_m,h_n),
\end{eqnarray*}
therefore,
\begin{eqnarray}\label{ine1}
\nonumber|\Hat{Y}_{m,n,a}(\Delta f((t,s),(x,y));x,y)| &\leq & \Hat{Y}_{m,n,a}\Bigg(|t-x||y-s|\Bigg(\left(1+\frac{|t-x|}{h_m}\right)\left(1+\frac{|s-y|}{h_n}\right)\omega_B(D_B f;h_m,h_n)\Bigg);x,y\Bigg)  \\
 &&+3\|D_B f\|\Hat{Y}_{m,n,a}(|t-x||y-s|;x,y),
\end{eqnarray}
since, 
\begin{eqnarray}\label{ine2}
|\Hat{{BY}}_{m,n}^{a}(f;x,y)-f(x,y)|\leq \Hat{Y}_{m,n,a}(|\Delta f((t,s),(x,y))|;x,y),
\end{eqnarray}
Upon using the inequalities (\ref{ine1}), (\ref{ine2}) and with the help of Cauchy-Schwarz inequality, we get 
\begin{eqnarray*}
|\Hat{{BY}}_{m,n}^{a}(f;x,y)-f(x,y)|&\leq & \Bigg\{\left(\Hat{Y}_{m,n,a}((t-x)^2(s-y)^2;x,y)\right)^{\frac{1}{2}}+h_m^{-1} \left(\Hat{Y}_{m,n,a}((t-x)^4(s-y)^2;x,y)\right)^{\frac{1}{2}}\\&&+h_n^{-1}\left(\Hat{Y}_{m,n,a}((t-x)^2(s-y)^4;x,y)\right)^{\frac{1}{2}}\\&&
+h_m^{-1}h_n^{-1}\left(\Hat{Y}_{m,n,a}((t-x)^4(s-y)^4;x,y)\right)^{\frac{1}{2}}\Bigg\}\omega_B(D_B f;h_m,h_n)\\&&
+3\|D_B f\|\left(\Hat{Y}_{m,n,a}((t-x)^2(s-y)^2;x,y)\right)^{\frac{1}{2}}\\
&=&\Bigg\{\sqrt{\Hat{Y}_{m,n,a}((t-x)^2;x,y)}\sqrt{\Hat{Y}_{m,n,a}((s-y)^2;x,y)}\\&&
+h_m^{-1}\sqrt{\Hat{Y}_{m,n,a}((t-x)^4;x,y)} \sqrt{\Hat{Y}_{m,n,a}((s-y)^2;x,y)}\\&& 
+h_n^{-1}\sqrt{\Hat{Y}_{m,n,a}((t-x)^2;x,y)} \sqrt{\Hat{Y}_{m,n,a}((s-y)^4;x,y)}  \\&&
+h_m^{-1}h_n^{-1}\sqrt{\Hat{Y}_{m,n,a}((t-x)^4;x,y)} \sqrt{\Hat{Y}_{m,n,a}((s-y)^4;x,y)}\Bigg\}\omega_B(D_B f;h_m,h_n)\\&&
+3\|D_B f\| \sqrt{\Hat{Y}_{m,n,a}((t-x)^2;x,y)}\sqrt{\Hat{Y}_{m,n,a}((s-y)^2;x,y)},
\end{eqnarray*}
Now using the inequalities (\ref{re2}), (\ref{re3}) and Lemma \ref{L5}, we have
\begin{eqnarray*}
|\Hat{{BY}}_{m,n}^{a}(f;x,y)-f(x,y)|&\leq &\Bigg\{ \sqrt{\frac{\lambda_x}{m}}\sqrt{\frac{\lambda_y}{n}}+h_m^{-1}\sqrt{\frac{M_x}{m^2}}\sqrt{\frac{\lambda_y}{n}}+h_n^{-1}\sqrt{\frac{\lambda_x}{m}} \sqrt{\frac{M_x}{m^2}}\\&&
+h_m^{-1}h_n^{-1}\sqrt{\frac{M_x}{m^2}}\sqrt{\frac{M_y}{n^2}}\Bigg\}\omega_B(D_B f;h_m,h_n)\\
&&+3\|D_B f\|  \sqrt{\frac{\lambda_x}{m}}\sqrt{\frac{\lambda_y}{n}},
\end{eqnarray*}
Upon considering $h_m^{-1}=\frac{1}{\sqrt{m}}$ and $h_n^{-1}=\frac{1}{\sqrt{n}}$, one can write
\begin{eqnarray*}
|\Hat{{BY}}_{m,n}^{a}(f;x,y)-f(x,y)|&\leq &\frac{1}{\sqrt{mn}}\Bigg\{\Bigg(\sqrt{\lambda_x \lambda_y}+\sqrt{M_x\lambda_y}+\sqrt{\lambda_x M_y}+\sqrt{M_x M_y}\Bigg) \omega_B\left( D_B f;\frac{1}{\sqrt{m}},\frac{1}{\sqrt{n}}\right)\\
&& +3\|D_B f\|\sqrt{\lambda_x \lambda_y} \Bigg\}\\
&= & \frac{1}{\sqrt{mn}} \Bigg\{(\sqrt{\lambda_x}+\sqrt{M_x})(\sqrt{\lambda_y}+\sqrt{M_y})\omega_B\left( D_B f;\frac{1}{\sqrt{m}},\frac{1}{\sqrt{n}}\right)+3\|D_B f\|\sqrt{\lambda_x \lambda_y}  \Bigg\}\\
&=& \frac{1}{\sqrt{mn}}\Bigg\{ M_1M_2\omega_B\left( D_B f;\frac{1}{\sqrt{m}},\frac{1}{\sqrt{n}}\right)+3M_3\|D_B f\|  \Bigg\},
\end{eqnarray*}
where $M_1=(\sqrt{\lambda_x}+\sqrt{M_x})$, $M_2=(\sqrt{\lambda_y}+\sqrt{M_y})$ and $M_3=\sqrt{\lambda_x \lambda_y}$ and $M_4=\max\{M_1M_2,M_3\}$, Hence the Inequality gives
\begin{eqnarray}
|\Hat{{BY}}_{m,n}^{a}(f;x,y)-f(x,y)|\leq \frac{M_4}{\sqrt{mn}}	\Bigg\{ 3M_3\|D_B f\|+\omega_B\left( D_B f;\frac{1}{\sqrt{m}},\frac{1}{\sqrt{n}}\right) \Bigg\}.
\end{eqnarray}
Hence, the proof is completed.
\end{proof}

To improve the measure of smoothness, a mixed $K$-functional is introduced (see \cite{BC}, \cite{CC}) and it is defined by 
\begin{eqnarray}
K_B(f;x_1,x_2)=\{\|f-g_1-g_2-h \|+x_1\| D_B^{2,0}g_1\|+x_2\| D_B^{0,2}g_2\|+x_1x_2\| D_B^{2,2}h\| \},
\end{eqnarray}
where $g_1\in D_B^{2,0}$, $g_2\in D_B^{0,2}$, $h\in D_B^{2,2}$ and $D_B^{i,j}$ represent the space of all functions $f\in C_B(X_b)$ for $0\leq i,j\leq 2$ having mixed partial derivatives $D_B^{\eta,\mu}f$ with $0\leq \eta\leq i $, $0\leq \mu\leq j $  defined by 
\begin{eqnarray}
D_xf(u,v)= D_B^{1,0}(f;u,v)=\underset{x\to u}\lim\frac{\Delta_xf([u,x];v)}{x-u},
\end{eqnarray}
\begin{eqnarray}
D_yf(u,v)= D_B^{0,1}(f;u,v)=\underset{y\to v}\lim\frac{\Delta_yf(u;[v,y])}{y-v},
\end{eqnarray}
\begin{eqnarray}
D_yD_xf(u,v)= D_B^{0,1}D_B^{1,0}(f;u,v)=\underset{y\to v}\lim\frac{\Delta_y(\Delta_x)f(u;[v,y])}{y-v},
\end{eqnarray}
\begin{eqnarray}
D_xD_yf(u,v)= D_B^{1,0}D_B^{0,1}(f;u,v)=\underset{x\to u}\lim\frac{\Delta_x(\Delta_y)f([u,x];v)}{x-u}.  
\end{eqnarray}
 where $\Delta_xf([u,x];v)=f(x,v)-f(u,v)$, $\Delta_yf(u;[v,y])=f(u,y)-f(u,v)$.
 
 \begin{theorem}
 Let $\Hat{{BY}}_{m,n}^{a}(f;x,y)$ be a GBS-type operator of $\Hat{Y}_{m,n,a}(f;x,y)$ for all $x,y\in X_b=[0,c]\times[0,d]$  and for each function $f\in C_B(X_b)$ with $m,n\in \mathbb{N}$, we have

 \begin{eqnarray}
 |\Hat{{BY}}_{m,n}^{a}(f;x,y)-f(x,y)|\leq 2K_B\left( f,\frac{\lambda_x}{m},\frac{\lambda_y}{n}\right).
\end{eqnarray}  
 \end{theorem}
 \begin{proof}
With the help of Taylor's formula for the function $g_1\in C_B^{2,0}(X_b)$, we obtain
\begin{eqnarray}
g_1(t,s)-g_1(x,y)=(t-x)D_B^{1,0}g_1(x,y)+\int\limits_x^t(t-\xi)D_B^{2,0}g_1(\xi,y)d\xi,
\end{eqnarray} 
Upon using the linearity and positivity properties of GBS-type operators and the definition of $\Hat{{BY}}_{m,n}^{a}(f;x,y)$, it gives 
\begin{eqnarray*}
|\Hat{{BY}}_{m,n}^{a}(g_1;x,y)-g_1(x,y)|&=& \Bigg|\Hat{Y}_{m,n,a}\Bigg( \int\limits_x^t(t-\xi)[D_B^{2,0}g_1(\xi,y)-D_B^{2,0}g_1(\xi,s)]d\xi;x,y\Bigg) \Bigg|\\
&\leq & \Hat{Y}_{m,n,a}\Bigg( \Bigg|\int\limits_x^t|(t-\xi)||D_B^{2,0}g_1(\xi,y)-D_B^{2,0}g_1(\xi,s)|d\xi;x,y\Bigg|\Bigg) \\
&\leq & \|D_B^{2,0}g_1\| \Hat{Y}_{m,n,a}((t-x)^2;x,y)\|< \|D_B^{2,0}g_1\| \frac{c}{m},
\end{eqnarray*}
 similarly, 
 \begin{eqnarray*}
 |\Hat{{BY}}_{m,n}^{a}(g_2;x,y)-g_2(x,y)|< \|D_B^{0,2}g_2\| \frac{d}{n},
 \end{eqnarray*}
let $g_2\in D_B^{0,2}$ then for $h\in\ D_B^{2,2}$, we have
\begin{eqnarray*}
h(t,s)-h(x,y)&=&(t-x)D_B^{1,0}h(x,y)+(s-y)D_B^{0,1}h(x,y)+(t-x)(s-y)D_B^{1,1}h(x,y)\\
&&+\int\limits_x^t(t-\xi)D_B^{2,0}h(\xi,y)d\xi+\int\limits_x^t(s-\phi)D_B^{0,2}h(x,\phi)d\phi+\int\limits_x^t(s-y)(t-\xi)D_B^{2,1}h(\xi,y)d\xi\\
&&+\int\limits_x^t(t-x)(s-\phi)D_B^{1,2}h(x,\phi)d\phi+ \int\limits_x^t\int\limits_y^s(t-\xi)(s-\phi)D_B^{2,2}h(\xi,\phi)d\xi d\phi.
\end{eqnarray*}  
By using the definition of the GBS-type operators $\Hat{{BY}}_{m,n}^{a}(f;x,y)$ of the defined operators (\ref{R1}), we have
 \begin{eqnarray}
 \Hat{{BY}}_{m,n}^{a}((t-x);x,y)=0,~ \Hat{{BY}}_{m,n}^{a}((s-y);x,y)=0,
 \end{eqnarray}
 in next step, we get 
 \begin{eqnarray*}
 |\Hat{{BY}}_{m,n}^{a}(h;x,y)-h(x,y)|&\leq &\Bigg|\Hat{BY}_{m,n,a} \Bigg(\int\limits_x^t\int\limits_y^s(t-\xi)(s-\phi)D_B^{2,2}h(\xi,\phi)d\xi d\phi;x,y \Bigg)\Bigg|\\ 
 &\leq & \Hat{BY}_{m,n,a} \Bigg(\int\limits_x^t\int\limits_y^s|(t-\xi)| |(s-\phi)|\Bigg|D_B^{2,2}h(\xi,\phi)\Bigg|d\xi d\phi;x,y \Bigg)\\
 &\leq & \frac{1}{4} \| D_B^{2,2}h\| \Hat{BY}_{m,n,a} ((t-x)^2(s-y)^2;x,y)\\
 &\leq & \| D_B^{2,2}h\|\frac{\lambda_x \lambda_y}{mn}.
\end{eqnarray*}
Now, 
\begin{eqnarray*}
|\Hat{{BY}}_{m,n}^{a}(f;x,y)-f(x,y)| &\leq & |(f-g_1-g_2-h)(x,y)| + \Bigg|\left(g_1-\Hat{{BY}}_{m,n}^{a}g_1 \right)(x,y) \Bigg|+\Bigg|\left(g_2-\Hat{{BY}}_{m,n}^{a}g_2 \right)(x,y) \Bigg| \\
&&+\Bigg|\left(h-\Hat{{BY}}_{m,n}^{a}h \right)(x,y) \Bigg|+ \Bigg|   \Hat{{BY}}_{m,n}^{a}((f-g_1-g_2-h);x,y)\Bigg|\\
&\leq & 2\|f-g_1-g_2-h \|+ \| D_B^{2,0}g_1\|\frac{\lambda_x}{m}+\| D_B^{0,2}g_2\|\frac{\lambda_y}{n}+\| D_B^{2,2}h\|\frac{\lambda_x \lambda_y}{mn},
\end{eqnarray*}
by taking infimum over for all $g_1\in C_B^{2,0}$, $g_2\in C_B^{0,2}$, $h\in C_B^{2,2}$, we get our desired result. 
\end{proof}
\section{Graphical approach and Convergence based discussion}
For validation of the results, the GBS-type operators are compared with the bivariate operators (\ref{R1}) and the rate of convergence is examined for finite sum over the interval $[0,1]$ as well as for infinite sum over the interval $[0,\infty)$ through  graphical representations along with their numerical approximation.

In this section, we discuss the behaviour of the operators with the  function $f(x,y)$ for particular values of $k_1,k_2$   and for an infinite series (i.e. for $k_1=0,1,\cdots, \infty$ and $k_2=0,1,\cdots, \infty$). Also, check the behaviour of the operators (\ref{R1}) and (\ref{R2}) by comparison.

\pagebreak
\begin{example}
Consider the function $f(x,y)=x\sin{\pi y}$ (green). For the particular value of $m=n=10$, $k_1=9=k_2$, the corresponding operators are represented by $\Hat{Y}_{10,10,a}(f;x,y)$(blue) and $\Hat{{BY}}_{10,10}^{a}(f;x,y)$(red) respectively. Upon considering the partitions as $x_0=0, x_1=\frac{1}{10}, \cdots, x_9=\frac{9}{10}$  of $[0,1]$ and $y_0=0,y_1=\frac{1}{10}, \cdots, y_9=\frac{9}{10}$ of $[0,1]$, the convergence approach of the operators $\Hat{Y}_{m,n,a}(f;x,y)$ and $\Hat{{BY}}_{m,n}^{a}(f;x,y)$ to the function and their comparison are shown in Figure \ref{F1}.

\begin{figure}[h!]
    \centering 
        \includegraphics[width=.32\textwidth]{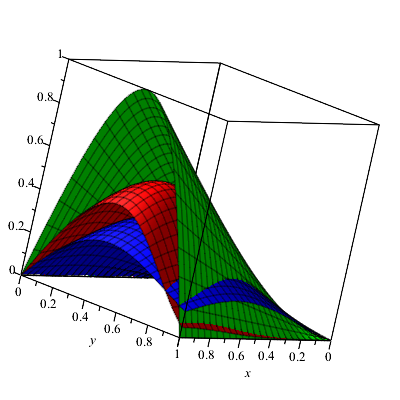}   
    \caption[Description in LOF, taken from~\cite{source}]{The comparison of the convergence approach of the operators $\Hat{Y}_{m,n,a}(f;x,y)$(blue) and $\Hat{{BY}}_{m,n}^{a}(f;x,y)$(red) to the function $f(x,y)$(green).}
    \label{F1}
\end{figure}
Now, we choose less numbers of partitions for the same function and for the same particular values of $m=n=10$. 
\begin{figure}[h!]
    \centering 
        \includegraphics[width=.32\textwidth]{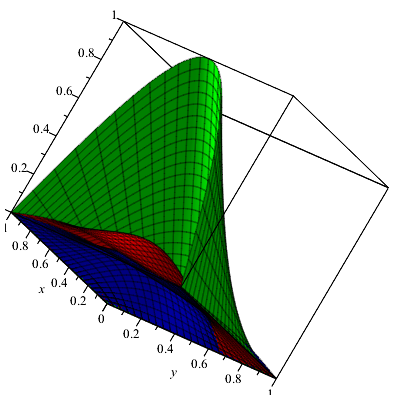}   
    \caption[Description in LOF, taken from~\cite{source}]{The comparison of the convergence approach  of the bivariate operators $\Hat{Y}_{m,n,a}(f;x,y)$ (blue) and $\Hat{{BY}}_{m,n}^{a}(f;x,y)$ (red) to the function $f(x,y)$ (green).}
    \label{F2}
\end{figure}
\end{example}

Here, we take the partitions within six terms like as $x_0=0,x_1=\frac{1}{10}, \cdots, x_5=\frac{5}{10}$ of $[0,1]$  and $y_0=0,y_1=\frac{1}{10}, \cdots, y_5=\frac{5}{10}$ of $[0,1]$ as shown in Figure \ref{F2}. It can be seen from Figure \ref{F2} that the error gap between the function and operators are maximum in Figure \ref{F2} rather than in Figure \ref{F1}.\\

Finally, it can be observed from Figures (\ref{F1}) and (\ref{F2}) that the accuracy approach of the GBS-type operators (\ref{R2}) to the function $f(x,y)$ is better than the bivariate operators (\ref{R1})  but it  depends on the number of partitions of $[0,1]$. By observing Figure \ref{F2} and Figure \ref{F1}, it can be seen that for large number of partitions i.e as the length of the partition be small, the approximation is better as compared to less number of partitions of the interval i.e for larger length of partitions. It can also be concluded that the approach of the operators to the function will be good upon using large number of partitions as compared to less numbers of partitions for the same interval. On other the hand, the approach of the GBS-type operators (\ref{R2}) is better than the bivariate operators (\ref{R1}). So, finally we can say that the convergence rate of the GBS-type operators is better than the convergence rate of the bivariate operators in any case.

\textbf{Remark:} In general, if we consider $[x_0,x_1],[x_1,x_2],\cdots,[x_{i-1},x_i]$ and $[y_0,y_1],[y_1,y_2],\cdots,[y_{j-1},y_j]$, are the sub-intervals of the $[0,1]$, provided each $x_i, y_j$ are the some form of $\frac{i}{m}, \frac{j}{n}$ respectively, where $i=1,2,\cdots, k_1,~j=1,2,\cdots, k_2$ while $k_1\leq m, k_2\leq n$, then the following concluding remarks can be obtained.\\
\textbf{Concluding Remark:} 
\begin{itemize}
\item{} If the number of sub-intervals are maximum i.e., the sub-length $x_i-x_{i-1}, y_j-y_{j-1}$ are small, then the approximation is good. 
\item{} If the number of sub-interval are minimum i.e., the sub-length $x_i-x_{i-1}, y_j-y_{j-1}$ are large, then the approximation is not good.
\end{itemize}

\textbf{Note:} In above both conditions, the approach of the GBS-type operators (\ref{R2}) to the function is better than the bivariate operators as defined by (\ref{R1}).

\begin{example}
Consider a function defined by $f(x,y)=x\sin{\pi y}$(green). For the particular value of $m=n=10$, the corresponding operators  $\Hat{Y}_{10,10,a}(f;x,y)$ and $\Hat{{BY}}_{10,10}^{a}(f;x,y)$ are shaded by blue and red colors respectively as given in Figure \ref{F3}. Here, it can be seen the approximation of the function defined by the operators (\ref{R1}), (\ref{R2}) and  the error determined by the GBS-type operators to the function is  minimum than the bivariate operators (\ref{R2}).
\begin{figure}[h!]
    \centering 
        \includegraphics[width=.32\textwidth]{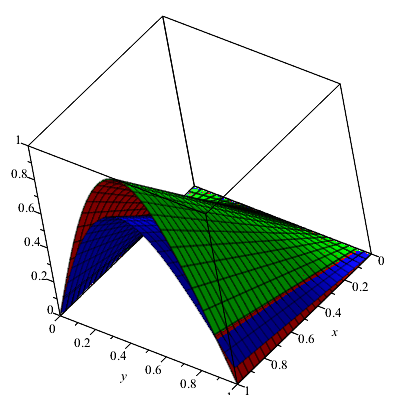}   
    \caption[Description in LOF, taken from~\cite{source}]{Comparison of the convergence for both bivariate operators $\Hat{Y}_{m,n,a}(f;x,y)$(blue) and $\Hat{{BY}}_{m,n}^{a}(f;x,y)$(red) to the function $f(x,y)$ (green).}
    \label{F3}
\end{figure}
\end{example}
\textbf{Concluding result:} From Figure \ref{F3}, it can be concluded that the convergence behavior of the GBS-type operators defined by (\ref{R2}) is better than the bivariate operators defined by (\ref{R1}).
\begin{example}
Consider a function $f(x,y)=\sin({x+y})$ (green). On choosing the value of $m=n=10, 15$ for the GBS-type operators, the corresponding GBS operators can be represented as $\Hat{{BY}}_{10,10}^{2}(f;x,y)$(blue), $\Hat{{BY}}_{15,15}^{2}(f;x,y)$(yellow). It can be observed from Figure \ref{F4} that the error becomes smaller as the value of $m$ and $n$ be increases. 
\begin{figure}[h!]
    \centering 
        \includegraphics[width=.32\textwidth]{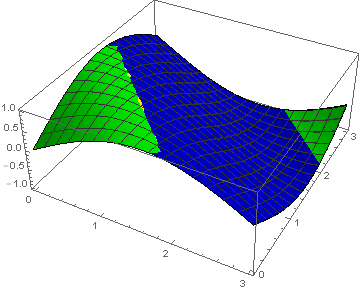}   
    \caption[Description in LOF, taken from~\cite{source}]{The convergence of the GBS-type operators $\Hat{{BY}}_{m,m}^{a}(f;x,y)$ to the function $f(x,y)$.}
    \label{F4}
\end{figure}

  From Figure \ref{F4}, it can be seen the convergence behaviour of the GBS-type operators with the small value of the parameters (as $m=n=10,15$) where as Figure \ref{F5} represents the convergence behaviour of the GBS-type operators $\Hat{{BY}}_{15,15}^{2}(f;x,y)$ for $m=n=15$ (yellow color) to the function in more closer form as compare to Figure \ref{F4}.
 

\begin{figure}[h!]
    \centering 
        \includegraphics[width=.32\textwidth]{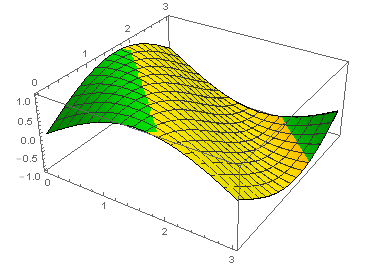}   
    \caption[Description in LOF, taken from~\cite{source}]{The convergence of the  GBS-type operator $\Hat{{BY}}_{15,15}^{2}(f;x,y)$(yellow) to the function $f(x,y)$(green).}
    \label{F5}
\end{figure}

\textbf{Concluding result:} It can be concluded from the graphical representations of the operators     
 $\Hat{{BY}}_{m,n}^{a}(f;x,y)$ and $\Hat{Y}_{m,n,a}(f;x,y)$ that the rate of convergence of the GBS-type operators (\ref{R2}) is better than the bivariate operators (\ref{R1}).

\end{example}
\subsection{Numerical approach}

Next, we discuss the absolute error of the GBS-type operators (\ref{R2}) as well as the bivariate operators (\ref{R1}) to the function $f(x,y)$ and compare these operators with their numerical errors at different points and for different values of $m,n$.\\

Let $G_{m,n}^a(f;x,y)=|\Hat{{BY}}_{m,n}^{a}(f;x,y)-f(x,y)|$  and $SY_{m,n,a}(f;x,y)=|\Hat{Y}_{m,n,a}(f;x,y)-f(x,y)|$, then the given Table \ref{Tab1}  represents the numerical approximations of  the GBS-type operators (\ref{R2}) and bivariate operators (\ref{R1}).

\begin{table}[h!]
\begin{tabular}{ |p{3cm}|p{5cm}|p{5cm}|  }
 \hline
 \multicolumn{3}{|c|}{Error in the approximation for GBS-type operators and bivariate operators  to the function $f(x,y)$} \\
 \hline
m=n      &  $SY_{m,n,a}(f;x,y)$ & $G_{m,n}^a(f;x,y)$ \\
 \hline
 10  &  0.00887548  & 0.0000358021\\
  \hline
15 &    0.00589734 & 0.0000160308 \\
 \hline
 25 &  0.0035283 & 5.80313$\times10^{-6}$ \\
 \hline
 50 &  0.00176016 & 1.45647$\times10^{-6}$ \\
 \hline
100 & 0.000879051 & 3.64803$\times10^{-7}$\\
\hline
\end{tabular}
\caption{A Comparison of the GBS-type operators and bivariate operators  to the function $f(x,y)$}\label{Tab1}
\end{table}

\textbf{Concluding remark:} From Table \ref{Tab1}, it can be observed that the approximation by the GBS-type operators (\ref{R2}) to the function is better than the bivariate operators (\ref{R1}).


\subsection{A comparison of the bivariate operators (\ref{R1}) with bivariate Kantorovich operators}

In this subsection, we show the graphical representation for the comparison of convregence of the bivariate operators (\ref{R1}) with the bivariate Kantorovich operators of Sz\'asz-Mirakjan. In 2006, Muraru \cite{MCV} gave a quantitative approximation of Kantorovich-Sz\'asz bivariate operators, defined as

\begin{eqnarray*}
\Hat{{K}}_{m,n}f : L_1([0,\infty) \times[0,\infty)) \to B([0,\infty)\times[0,\infty)),~(m,n)\in \mathbb{N}\times\mathbb{N};
\end{eqnarray*}

\begin{eqnarray}\label{kan}
\Hat{{K}}_{m,n}(f;x,y)=mne^{-mx-ny}\sum\limits_{k_1=0}^\infty\sum\limits_{k_2=0}^\infty \frac{(mx)^{k_1}}{k_1!}\frac{(nx)^{k_2}}{k_2!}\int\limits_{\frac{k_1}{m}}^{\frac{k_1+1}{m}}\int\limits_{\frac{k_2}{n}}^{\frac{k_2+1}{n}}f(u,v)~du dv.
\end{eqnarray}
There are following computational examples, which represent the comparison.

\begin{example}
Let the function $f(x)=x^2y(x-1)\sin(2\pi y)$ (green), for all $0\leq x, y\leq 2$ and choose the value of $m, n=10$, for which the bivariate  operators (yellow) defined by (\ref{R1}) show the better rate of convergence than the Kantorovich-Sz\'asz bivariate operators $\Hat{K}_{m,n}(f;x,y)$(red) defined by (\ref{kan}), graphical represenetation can be seen by the Figure \ref{F1}.

\begin{figure}[h!]
    \centering 
    \includegraphics[width=.42\textwidth]{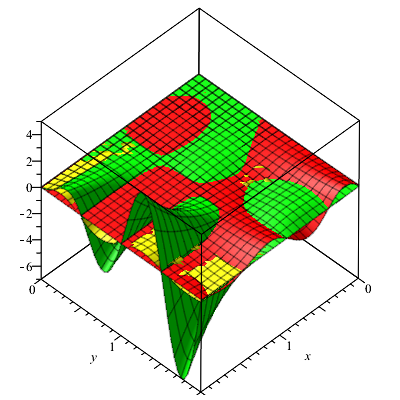}   
    \caption[Description in LOF, taken from~\cite{source}]{The comparison of the convergence of the operators $\Hat{Y}_{m,n,a}(f;x,y)$ (yellow) and $\Hat{K}_{m,n}(f;x,y)$ (red) to the function $f(x)$ (green)}
    \label{F1}
\end{figure}

\end{example}

\begin{example}
Consider the function $f(x)=x^2y\cos(\pi y)$ (green), for all $0\leq x, y\leq 4$ and choose $m, n=10$, for which the bivariate operators (yellow) defined by (\ref{R1}) present the
 better rate of convergence than the bivariate Kantorovich operators $\Hat{K}_{m,n}(f;x,y)$ (red)  defined by (\ref{kan}), the graphical representation is illustrated by Figure \ref{F2}.

\begin{figure}[h!]
    \centering 
    \includegraphics[width=.42\textwidth]{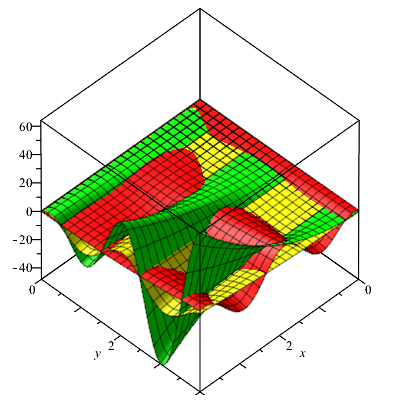}   
    \caption[Description in LOF, taken from~\cite{source}]{The comparison of the convergence of the operators $\Hat{Y}_{m,n,a}(f;x,y)$ (yellow) and $\Hat{K}_{m,n}(f;x,y)$ (red) to the function $f(x)$ (green)}
    \label{F2}
\end{figure}
\end{example}

\begin{example}
Let the function $f(x)=y^2\cos(2\pi x)$ (green), for all $0\leq x, y\leq 4$ and consider $m, n=20$, for which the graphical representation of the bivariate operators $\Hat{Y}_{m,n}(f;x,y)$  (yellow) defined by (\ref{R1}) and the bivariate Kantorovich operators $\Hat{K}_{m,n}(f;x,y)$ (red)  defined by (\ref{kan}) is illustrated in Figure \ref{F3}.
%
%

\begin{figure}[h!]
    \centering 
    \includegraphics[width=.42\textwidth]{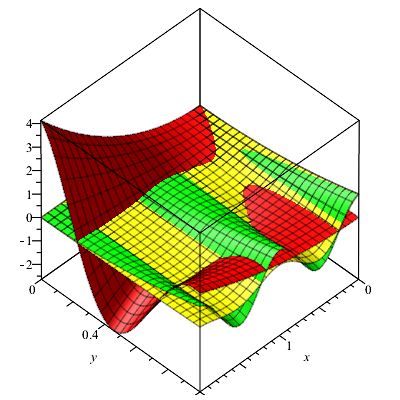}   
    \caption[Description in LOF, taken from~\cite{source}]{The comparison of the convergence of the operators $\Hat{Y}_{m,n,a}(f;x,y)$ (yellow) and $\Hat{K}_{m,n}(f;x,y)$ (red) to the function $f(x)$ (green)}
    \label{F3}
\end{figure}
\end{example}

\textbf{Concluding Remark:}  By the above figures (\ref{F1}, \ref{F2}, \ref{F3}), we can say the rate of convergence of the bivaraite operators $\Hat{Y}_{m,n,a}(f;x,y)$ (\ref{R1}) is better than biavraite Kantorovich operators $\Hat{K}_{m,n}(f;x,y)$  defined by (\ref{kan}).


\subsection{ Comparison of associated GBS operators with the GBS-type operators of an infinite sum}
In 2008, Pop \cite{PO3} introduced an associated GBS-type operators of the linear positive operators defined by an infinite sum, so called GBS operators of Mirakjan-Favard-Sz\'{a}sz and gave an approximation of the functions considered to be $B$-continuous and $B$-differentiable. The  defined associated GBS operators can be stated as:


Let $m,n\in\mathbb{N}$, the operators  $UL_{m,n}^*:E(I\times I)\to F(J\times J)$ are defined for any function $f\in E(I\times I)$ and for $(x,y)\in J\times J$ such that


\begin{eqnarray}\label{s1}
UL^*_{m,n}(f;x,y)&=& \sum\limits_{k_1=0}^{\infty}\sum\limits_{k_2=0}^{\infty}\psi_{m,k_1}(x)\psi_{n,k_2}(y)[f(x_{m,k_1},y)+f(x,x_{n,k_2})-f(x_{m,k_1},x_{n,k_2})], 
\end{eqnarray}
%
%
%

where $((x_{m,k_1})_{k_1\in\mathbb{N}_0})_{m\geq1}, ((x_{n,k_2})_{k_2\in\mathbb{N}_0})_{n\geq1}$ are the sequences of nodes and the functions $\psi_{m,k_1}:I\to\mathbb{R}, \psi_{n,k_2}:J\to\mathbb{R}$ with the properties, $\psi_{m,k_1}\geq 0, \psi_{n,k_2}\geq 0$, where $I,J\subset\mathbb{R}, I\cap J\neq\phi$. 
The above GBS-modification operators (\ref{s1}) are the GBS-form of the operators $L^*$-type operators \cite{PO3} and are given by:
\begin{eqnarray}\label{s2}
L^*_{m,n}(f;x,y)&=& \sum\limits_{k_1=0}^{\infty}\sum\limits_{k_2=0}^{\infty}\psi_{m,k_1}(x)\psi_{n,k_2}(y)f(x_{m,k_1},x_{n,k_2}),~~~(x,y)\in J\times J,
\end{eqnarray}
where $m,n\in\mathbb{N}$, $f\in E(I\times I)$ and $L_{m,n}^*:E(I\times I)\to F(J\times J)$.\\

For the particular case, Pop \cite{PO3} determined the convergence properties for the GBS operators of Mirakjan-Favard-Sz\'{a}sz. Here, if $\psi_{m,k_1}(x)=\frac{k_1}{m}, \psi_{n,k_2}(x)=\frac{k_2}{n}$ and $\psi_{m,k_1}=e^{-mx}\frac{(mx)^{k_1}}{k_1!}, \psi_{n,k_2}=e^{-ny}\frac{(ny)^{k_2}}{k_2!}$ then for $f\in C([0,\infty)\times[0,\infty))$, the above operators (\ref{s1}) can be reduced to GBS operators of  Mirakjan-Favard-Sz\'{a}sz, which can be defined as follows:

\begin{eqnarray}\label{s2}
US^*_{m,n}(f;x,y)&=& \sum\limits_{k_1=0}^{\infty}\sum\limits_{k_2=0}^{\infty}e^{-mx-ny}\frac{(mx)^{k_1}}{k_1!}\frac{(ny)^{k_2}}{k_2!}\left[f\left(\frac{k_1}{m},y\right)+f\left(x,\frac{k_2}{n}\right)-f\left(\frac{k_1}{m},\frac{k_1}{m}\right)\right]. 
\end{eqnarray}

This subsection is very crucial from the comparison point of view of the GBS type operators as defined by (\ref{R2}) with the GBS operators of the  Mirakjan-Favard-Sz\'{a}sz type (\ref{s2}), which is shown by the following examples.

\begin{example}
Let the function be defined by $f(x,y)=e^{x+y}$(green). A comparison for the convergence  of the GBS-type operators $\Hat{{BY}}_{m,n}^{a}(f;x,y)$ (red) with the GBS operators of Mirakjan-Favard-Sz\'{a}sz $US^*_{m,n}(f;x,y)$ (black) to the function $f(x,y)$  is illustrated in Figure \ref{F7} for $m=n=2$. It can be observed that the GBS-type operators defined by (\ref{R2}) have a better rate of convergence than the GBS operators of Mirakjan-Favard-Sz\'{a}sz as defined by (\ref{s2}).
\begin{figure}[h!]
    \centering 
        \includegraphics[width=.52\textwidth]{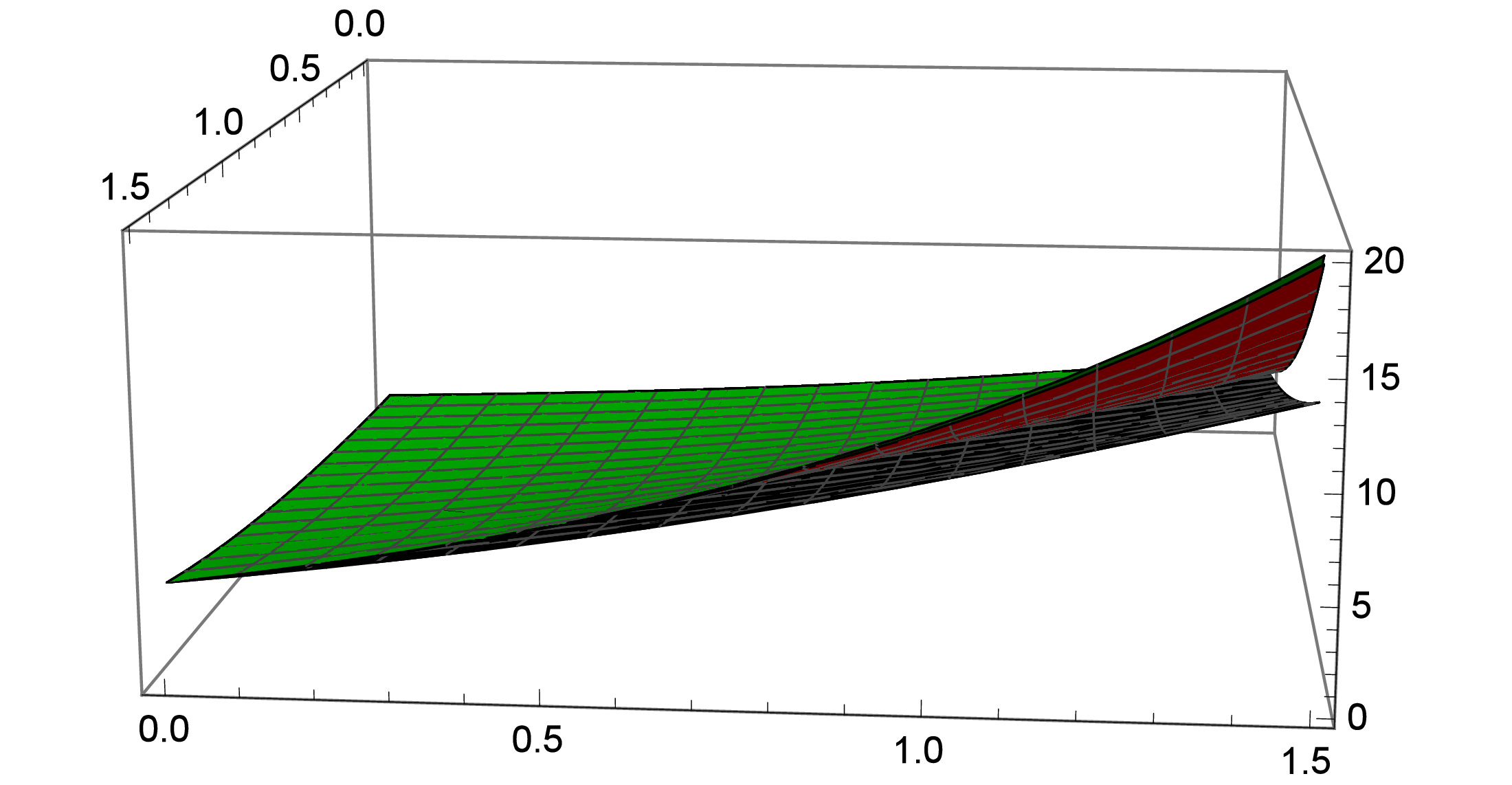}   
    \caption[Description in LOF, taken from~\cite{source}]{A comparison of the rate of convergence of the GBS-type operators $\Hat{{BY}}_{m,m}^{a}(f;x,y)$ and GBS operators of the Mirakjan-Favard-Sz\'{a}sz to the function $f(x,y)$.}
    \label{F7}
\end{figure}
\end{example}
\begin{example}
For the same function $f(x,y)=e^{x+y}$ and at a certain point $(x,y)$, the error estimation of the GBS-type operators $\Hat{{BY}}_{m,n}^{a}(f;x,y)$ and  GBS operators of Mirakjan-Favard-Sz\'{a}sz $US^*_{m,n}(f;x,y)$ has been computed in Table \ref{Tab2}.
\begin{table}[h!]
\begin{tabular}{ |p{3cm}|p{6cm}|p{6cm}|  }
 \hline
 \multicolumn{3}{|c|}{Error in the approximation  for $\Hat{{BY}}_{m,n}^{a}(f;x,y)$ and $US^*_{m,n}(f;x,y)$  to the function $f(x,y)$} \\
 \hline
m=n      &  $|US^*_{m,n}(f;x,y)-f(x,y)|$ & $|\Hat{{BY}}_{m,n}^{a}(f;x,y)-f(x,y)|$ \\
 \hline
 10  &  3.28277$\times 10^{-5}$  & 3.00798$\times 10^{-6}$\\
  \hline
20 &    7.91371$\times10^{-6}$ & 7.35189$\times10^{-7}$ \\
 \hline
 50 &  1.23907$\times10^{-6}$ & 1.16048$\times10^{-7}$ \\
 \hline
 100 &  3.07549$\times10^{-7}$ & 2.88814$\times10^{-8}$ \\
 \hline
\end{tabular}
\caption{A Comparison of the GBS-type operators $\Hat{{BY}}_{m,n}^{a}(f;x,y)$ and GBS operators of Mirakjan-Favard-Sz\'{a}sz $US^*_{m,n}(f;x,y)$  to the function $f(x,y)$}.\label{Tab2}
\end{table}

\end{example}
\textbf{Concluding Remark:} 
It can be concluded from Table \ref{Tab2} that the error arising in the approximation at a certain point by  GBS-type operators defined by (\ref{R2}) to the function  is much smaller than the GBS operators of Mirakjan-Favard-Sz\'{a}sz as defined by (\ref{s2}). Hence our GBS-type operators have a better rate of convergence than the GBS operators of Mirakjan-Favard-Sz\'{a}sz type.


\begin{thebibliography}{10}
\bibitem{g7} Abel U. On the asymptotic approximation with bivariate operators of Bleimann, Butzer, and Hahn. Journal of Approximation Theory. 1999 Mar 1;97(1):181-98.
\bibitem{g10} Agrawal PN, İspir N. Degree of approximation for bivariate Chlodowsky Sz$\acute{\text{a}}$sz-Charlier type operators. Results in Mathematics. 2016 Jun 1;69(3-4):369-85.
\bibitem{g8} Acar T, Aral A, Mohiuddine SA. Approximation by bivariate $(p, q)$-Bernstein-Kantorovich operators. Iranian Journal of Science and Technology, Transactions A: Science. 2018 Jun 1;42(2):655-62.
\bibitem{1} Agrawal PN, İspir N. Degree of approximation for bivariate Chlodowsky-Sz$\acute{\text{a}}$sz-Charlier type operators. Results in Mathematics. 2016 Jun 1;69(3-4):369-85.
\bibitem{2} Agrawal PN, Ispir N, Kajla A. GBS operators of Lupaş-Durrmeyer type based on P$\acute{\text{o}}$lya distribution. Results in Mathematics. 2016 Jun 1;69(3-4):397-418.
\bibitem{BK1} B$\ddot{\text{o}}$gel K.  Mehrdimensionale Differentiation von Funktionen mehrerer reellerVer$\ddot{\text{a}}$nderlichen. J. Reine Angew. Math. 170, 197-217 (1934).
\bibitem{BK2} B$\ddot{\text{o}}$gel K., $\ddot{\text{U}}$ber die mehrdimensionale differentiation, integration und beschr$\ddot{\text{a}}$nkte variation. J. Reine Angew. Math. 173, 5-29 (1935).
\bibitem{BCB1} Badea, C., Badea, I., Gonska, H.H., A test function theorem and approximation by pseudopolynomials. Bull. Aust. Math. Soc. 34, 53-64 (1986).
\bibitem{BCC}  Badea C, Cottin C. Korovkin-Type Theorems for Generalized Boolean Sum Operators, Approximation Theory (Kecskem$\acute{\text{e}}$t 1900), Colloq. Math. Soc. J$\acute{\text{a}}$nos Bolyai, North-Holland, Amsterdam, 58, 51-68 (1991).
\bibitem{BCB2} Badea C, Badea I, Cottin C, Gonska HH. Notes on the degree of approximation of $B$-continuous and $B$-differentiable functions. J. Approx. Theory Appl. 4, 95-108 (1988).	
\bibitem{BCGH1} Badea C, Cottin C, Gonska HH. B$\ddot{\text{o}}$gel Functions, Tensor Products and Blending Approximation. Mathematische Nachrichten. 1995;173(1):25-48.
\bibitem{BC} Badea C. $K$-functionals and moduli of smoothness of functions defined on compact metric spaces. Computers and Mathematics with Applications. 1995 Sep 1;30(3-6):23-31.
\bibitem{g2} B$\breve{\text{a}}$rbosu D. Some generalized bivariate Bernstein operators. Miskolc Mathematical Notes. 2000;1(1):3-10.
\bibitem{BP1} Baliarsingh P, Dutta S. On the classes of fractional order difference sequence spaces and their matrix transformations. Applied Mathematics and Computation. 2015 Jan 1;250:665-74.
\bibitem{BP2} Baliarsingh P, Nayak L. A note on fractional difference operators. Alexandria Engineering Journal. 2018 Jun 1;57(2):1051-4.
\bibitem{BP3} Baliarsingh P. On a fractional difference operator. Alexandria Engineering Journal. 2016 Jun 1;55(2):1811-6.
\bibitem{BP4} Baliarsingh P, Dutta S. On an explicit formula for inverse of triangular matrices. Journal of the Egyptian Mathematical Society. 2015 Jul 1;23(2):297-302.
\bibitem{BD1} Barbosu D. GBS operators of Schurer-Stancu type. Annals of the University of Craiova-Mathematics and Computer Science Series. 2003 Jan 1;30:34-9.
\bibitem{BD2} B$\breve{\text{a}}$rbosu D, B$\breve{\text{a}}$rbosu M. On the sequence of GBS operators of Stancu-type. Informatica. 2002;18(1):1-6.
\bibitem{3} B$\breve{\text{a}}$rbosu D, Muraru CV. Approximating B-continuous functions using GBS operators of Bernstein–Schurer–Stancu type based on q-integers. Applied Mathematics and Computation. 2015 May 15;259:80-7.
\bibitem{6} B$\breve{\text{a}}$rbosu D, Acu AM, Muraru CV. On certain GBS-Durrmeyer operators based on $ q $-integers. Turkish Journal of Mathematics. 2017 Apr 3;41(2):368-80.
\bibitem{CC} Cottin C. Mixed $K$-functionals: A measure of smoothness for blending-type approximation. Mathematische Zeitschrift. 1990 Dec 1;204(1):69-83.
\bibitem{7} Chauhan R, Ispir N, Agrawal PN. A new kind of Bernstein-Schurer-Stancu-Kantorovich-type operators based on $q$-integers. Journal of inequalities and applications. 2017 Dec 1;2017(1):50.
\bibitem{DSM} Das G, Srivastava VP, Mohapatra RN. On absolute summability factors of infinite series. J. Indian Math. Soc. 1967;31:189-200.
\bibitem{g3} Derriennic MM. On multivariate approximation by Bernstein-type polynomials. Journal of approximation theory. 1985 Oct 1;45(2):155-66.
\bibitem{g4} Do$\breve{\text{g}}$ru O, Gupta V. Korovkin-type approximation properties of bivariate $q$-Meyer-König and Zeller operators. Calcolo. 2006 Mar 1;43(1):51-63.
\bibitem{g6} Ditzian Z, Zhou X. Optimal approximation class for multivariate Bernstein operators. Pacific Journal of Mathematics. 1993 Mar 1;158(1):93-120.
\bibitem{g12} Duman O, Erku\c{s} E, Gupta V. Statistical rates on the multivariate approximation theory. Mathematical and computer modelling. 2006 Nov 1;44(9-10):763-70.
\bibitem{new} Dobrescu E, Matei I. The approximation by Berntein type polynomials of bidimensionally continuous functions. An. Univ. Timisoara Ser. Sti. Mat.-Fiz. no. 4, pp. 85-90 (1966) (Romanian).

\bibitem{17} Farcas MD. About approximation of $B$-continuous and $B$-differentiable functions of three variables by GBS operators of Bernstein type. Creat Math Inform 2008; 17: 20-27.
\bibitem{18} Farcas MD. About approximation of $B$-continuous functions of three variables by GBS operators of Bernstein type on a tetrahedron. Acta Univ Apulensis Math Inform 2008; 16: 93-102.
\bibitem{19} Farcas MD. About approximation of $B$-continuous and $B$-differential functions of three variables by GBS operators of Bernstein-Schurer type. Bul \c{S}tiin\c{t} Univ Politeh Timis Ser Mat Fiz 2007; 52: 13-22.
\bibitem{KMP} Kadak U, Mishra VN, Pandey S. Chlodowsky type generalization of $(p, q)$-Sz\'{a}sz operators involving Brenke type polynomials. Revista de la Real Academia de Ciencias Exactas, F\'{i}sicas y Naturales. Serie A. Matem\'{a}ticas. 2018 Oct 1;112(4):1443-62.
\bibitem{MCV} Muraru CV. Kantorovich-Szász bivariate operators, Stud.Cercet.Stiint., Ser.Mat., 16 (2006), Supplement Proceedings of ICMI 45, Bacau, Sept.18-20, 2006, pp. 169-176.
\bibitem{GM} Mirakjan GM. Approximation of continuous functions with the aid of polynomials $e^{-nx}\sum\limits_{k=0}^{m}C_{k,n}\chi^k,$ Dokl. Akad. Nauk SSSR, 31 (1941), pp. 201-205.

\bibitem{MD1} Miclaus D. On the GBS Bernstein-Stancu's type operators. Creat. Math. Inform. 2013;22(1):73-80.
\bibitem{RYM1} Mishra VN, Yadav R. Some estimations of summation-integral-type operators. Tbilisi Mathematical Journal. 2018;11(3):175-91.
\bibitem{MN} Mursaleen M, Nasiruzzaman M. Dunkl generalization of Kantorovich type Sz\'{a}sz–Mirakjan operators via $q$-calculus. Asian-European Journal of Mathematics. 2017 Dec 8;10(04):1750077.
\bibitem{MI} Manav. N, Ispir N. Approximation by blending type operators based on Sz$\acute{\text{a}}$sz-Lupa\c{s} basis functions. General Mathematics Vol. 24, No. 1-2 (2016), 105-119.
\bibitem{g9} \"{O}rkc\"{u} M. Approximation properties of bivariate extension of $q$-Sz$\acute{\text{a}}$sz-Mirakjan-Kantorovich operators. Journal of Inequalities and Applications. 2013 Dec 1;2013(1):324.
\bibitem{g11} \"{O}rkc\"{u} M. Sz$\acute{\text{a}}$sz-Mirakyan-Kantorovich Operators of Functions of Two Variables in Polynomial Weighted Spaces. In Abstract and applied analysis 2013 (Vol. 2013). Hindawi.
\bibitem{OS} O. Sz\'{a}sz, Generalization of S. Bernstein’s polynomials to the infinite interval, J. Res. Nat. Bur. Standards Sect. B. 45 (1950), pp. 239-245.
\bibitem{PO1} Pop OT. Approximation of $B$-differentiable functions by GBS operators. Anal. Univ. Oradea, Fasc. Matem. 2007;14:15-31.
\bibitem{PO2} Pop OT, Barbosu D. GBS operators of Durrmeyer-Stancu type. Miskolc Math Notes. 2008 Jan 1;9:53-60.
\bibitem{PO3} Pop OT. Approximation of $B$-continuous and $B$-differentiable functions by GBS operators defined by infinite sum. J. Inequal. Pure Appl. Math. 2009;10(1):1-7. 
\bibitem{RS} Ray BK, Sahoo AK. Application of the absolute Euler method to some series related to Fourier series and its conjugate series. In Proceedings of the Indian Academy of Sciences-Mathematical Sciences 1996 Feb 1 (Vol. 106, No. 1, pp. 13-38). Springer India.
\bibitem{4} Sidharth M, Ispir N, Agrawal PN. GBS operators of Bernstein-Schurer-Kantorovich type based on $q$-integers. Applied Mathematics and Computation. 2015 Oct 15;269:558-68.
\bibitem{5} Sidharth  M, Ispir N, Agrawal PN. Approximation of $B$-continuous and $B$-differentiable functions by GBS operators of $q$-Bernstein-Schurer-Stancu type, Turk J Math (2016) 40: 1298-1315, doi:10.3906/mat-1509-66. 
\bibitem{RYV1} Yadav R, Meher R, Mishra V.N., Results on bivariate Sz$\acute{\text{a}}$sz-Mirakjan type operators in  polynomial weight spaces, ArXiv preprint  arXiv:1911.08898v1 [math.NA] .


\bibitem{YMM} Yadav R, Meher R, Mishraa VN. Quantitative estimations of bivariate summation-integral-type operators. Mathematical Methods in the Applied Sciences. 2019 Aug 29 ;1-20. https://doi.org/10.1002/mma.5824.






\end{thebibliography}
\end{document}